\documentclass[a4paper, 11pt]{article}

\usepackage[utf8]{inputenc}
\usepackage{amsmath}
\usepackage{amsthm}
\usepackage{amsfonts}
\usepackage{amssymb}
\usepackage{amstext}
\usepackage{mathrsfs}
\usepackage{yhmath}
\usepackage{dsfont}
\usepackage{graphicx}
\usepackage{color}
\usepackage[colorlinks]{hyperref}
\usepackage{epigraph}
\usepackage{fullpage}
\usepackage[left=3.5cm,top=2cm,bottom=3.5cm,right=3cm,nohead,nofoot]{geometry}


\title{ A note on a Poissonian functional and a $q$-deformed Dufresne identity }
\author{Reda \textsc{Chhaibi}        \footnote{\texttt{reda.chhaibi@math.univ-toulouse.fr}} } %\footnotemark
\date{}

\allowdisplaybreaks[4]

\DeclareMathOperator{\Var}{Var}

\begin{document}

\newcommand\half{\frac{1}{2}}
\newcommand\eqlaw{\stackrel{\Lc}{=}}

%Ordinals
\newcommand\N{{\mathbb N}}
\newcommand\Z{{\mathbb Z}}
\newcommand\Q{{\mathbb Q}}
\newcommand\R{{\mathbb R}}
\newcommand\C{{\mathbb C}}

%Probability
\renewcommand\P{{\mathbb P}}
\newcommand\E{{\mathbb E}}

%Caligraphed letters
\newcommand\Ac{{\mathcal A}}
\newcommand\Bc{{\mathcal B}}
\newcommand\Cc{{\mathcal C}}
\newcommand\Dc{{\mathcal D}}
\newcommand\Ec{{\mathcal E}}
\newcommand\Fc{{\mathcal F}}
\newcommand\Gc{{\mathcal G}}
\newcommand\Hc{{\mathcal H}}
\newcommand\Ic{{\mathcal I}}
\newcommand\Kc{{\mathbb K}}
\newcommand\Lc{{\mathcal L}}
\newcommand\Oc{{\mathcal O}}
\newcommand\Pc{{\mathcal P}}
\newcommand\Qc{{\mathcal Q}}
\newcommand\Rc{{\mathcal R}}
\newcommand\Sc{{\mathcal S}}
\newcommand\Tc{{\mathcal T}}
\newcommand\Uc{{\mathcal U}}
\newcommand\Zc{{\mathcal Z}}

\maketitle

%Environnments
\setlength{\footskip}{2cm}

\numberwithin{equation}{section}
\numberwithin{figure}{section}

\newtheorem{thm}{Theorem}[section]
\newtheorem{proposition}[thm]{Proposition}
\newtheorem{corollary}[thm]{Corollary}
\newtheorem{question}[thm]{Question}
\newtheorem{conjecture}[thm]{Conjecture}
\newtheorem{definition}[thm]{Definition}
\newtheorem{example}[thm]{Example}
\newtheorem{lemma}[thm]{Lemma}
\newtheorem{properties}[thm]{Properties}
\newtheorem{property}[thm]{Property}
\newtheorem{rmk}[thm]{Remark}

\newcommand{\Card}[1]{\textup{Card($#1$)} \xspace}

\begin{abstract}
In this note, we compute the Mellin transform of a Poissonian exponential functional, the underlying process being a simple continuous time random walk. It shows that the Poissonian functional can be expressed in term of the inverse of a $q$-gamma random variable.

The result interpolates between two known results. When the random walk has only positive increments, we retrieve a theorem due to Bertoin, Biane and Yor. In the Brownian limit ($q \rightarrow 1^-$), one recovers Dufresne's identity involving an inverse gamma random variable. Hence, one can see it as a $q$-deformed Dufresne identity.
\end{abstract}

% --------------------------------------------------------------------
% Table of contents
%\tableofcontents

% --------------------------------------------------------------------
\section{Introduction}

Let $P^+$ and $P^-$ be two standard Poisson processes with respective intensities $1$ and $z = q^\mu \in \left[0, 1\right)$. In all the following, we assume $0 < q < 1$ and $\mu>0$. Then $\zeta_t := P_t^+ - P_t^-$ is a compound Poisson process with jumps $+1$ or $-1$, and drifting to $+\infty$. We investigate the law of the perpetuity
\begin{align}
\label{eq:def_perpetuity}
I^{(q, z)} := & \int_0^\infty q^{\zeta_s} ds \ ,
\end{align}
where $0<q<1$. We can also rewrite $I^{(q, z)}$ as
\begin{align}
\label{eq:def_perpetuity2}
I^{(q, z)} = & \frac{1}{1+z}\sum_{n=0}^\infty q^{B_n} \varepsilon_n \ ,
\end{align}
where $\left( \varepsilon_n; n \geq 1 \right)$ are i.i.d exponential random variables and $B_n$ is a standard random walk drifting to $+\infty$. More precisely
$$ \P\left( B_{n+1}-B_n = 1 \right) = 1-\P\left( B_{n+1}-B_n = -1 \right) = \frac{1}{1+z} > \half \ .$$

In the denomination of \cite{bib:by}, the random variable of interest is an exponential functional associated to the L\'evy process $ -\log(q) \zeta_t $ .
We have:
$$ \E\left( e^{-s \log(q) \zeta_t} \right) = \exp\left( t \psi(s) \right) $$
where the L\'evy-Khintchine exponent is:
\begin{align}
\label{eq:levy_khintchine_exp}
\psi(s) = & \left( q^{-s} - 1 \right) + z\left( q^s - 1 \right) \ .
\end{align}

The subject of exponential functionals has been extensively studied (\cite{bib:by}, \cite{bib:cpy97}, \cite{bib:bl13}) with a focus on many possible aspects such as their continuity properties \cite{bib:ls11} and their Mellin transforms \cite{bib:hy13}. From a theoretical perspective, it is known since Kesten \cite{bib:kesten} that such perpetuities appear naturally when analyzing products of random matrices. More precisely, perpetuities appear as matrix coefficients of random walks in lower (or upper) triangular matrices. The study of such objects is in fact our original motivation. In insurance and mathematical finance, the perpetuity \eqref{eq:def_perpetuity} can be understood as the total discounted cash-flow from a perpetual bond, when interest rates are random. Also, the discussion in \cite{bib:bby} shows a connection to models in DNA replication and transmission protocols. With such applications in mind, corollary \ref{corollary:cv_gamma} allows to simulate $I^{(q, z)}$ with little effort.

Our main theorem \ref{thm:main} gives an explicit formula for the Mellin transform of $I^{(q, z)}$ and its density. This allows to recognize in corollary \ref{corollary:cv_gamma} a Wiener-Hopf factorization: $I^{(q, z)}$ has the same law as the product of two independent random variables, one of which is the inverse of a $q$-gamma random variable - to be defined later in the text.

On the one hand, notice that if we specialize $z$ to zero, we obtain
$$ I^{(q)} := I^{(q, z=0)} = \sum_{n=0}^\infty q^{n} \varepsilon_n \ ,$$
which is exactly the object studied by Bertoin, Biane and Yor in \cite{bib:bby}. In that respect, our result is a refinement of their theorem 1. A purely analytic description of that case has been given by \cite{bib:berg}.

On the other hand, consider the compound Poisson process $\zeta_{t}$ properly rescaled in time and space thanks to the parameter $q$. Thanks to Donsker's invariance principle, it will converge weakly to a Brownian motion as $q \rightarrow 1$, in the uniform topology. The exponential functionals will naturally converge too, although the claim requires some work because the exponential functionals are not continuous on the sample path space. This way, one recovers Dufresne's identity involving the inverse of the usual gamma random variable.

In the light of what has been just said, corollary \ref{corollary:cv_gamma} is a $q$-deformed Dufresne identity. As we will explain, the definition of the $q$-gamma distribution from $q$-calculus can be subject to controversy since many candidates qualify as $q$-deformations. Therefore, our result shows that the $q$-gamma distribution appears in relation to the random walk exactly in the same way as the gamma distribution appears in relation to Brownian motion. For possible multi-dimensional generalisations using certain random walks on solvable Lie groups, see \cite{bib:chh14b}.\\

It is worth mentioning that the L\'evy-Khintchine exponent \eqref{eq:levy_khintchine_exp} has the following Wiener-Hopf factorization (see Chapter 6 in \cite{bib:kyprianou})
\begin{align}
\label{eq:wiener_hopf}
\psi(s) = -\phi_+(s) \phi_-(-s) & \textrm{ with } \phi_+(s) = q^{-s} - 1, \ \phi_-(s) = z q^{-s} - 1 \ .
\end{align}
The functions $\phi_+$ and $\phi_-$ are exponents of subordinators and more precisely the ascending and descending ladder height processes associated to $-\zeta \log q$. One recognizes immediately that $\phi_+$ is the exponent of a rescaled standard Poisson process while $\phi_-$ is the exponent of a rescaled Poisson process with rate $z$ and killing rate $1-z$. The rescaling in space is $-\log q$.

Let $\xi$ be a L\'evy process and let $I_\xi := \int_0^\infty e^{-\xi_s} ds$ be its exponential functional. Under fairly general assumption on $\xi$, \cite{bib:PPS12} and \cite{bib:PS12} prove that $I_\xi$ enjoys a remarkable factorization into the product of two independent exponential functionals $I_{H}$ and $I_Y$. Since this factorization \emph{in law} is a shadow of the Wiener-Hopf factorization \eqref{eq:wiener_hopf} of the Lévy exponent, it is referred to by the same denomination. In their construction, $H$ is the ascending ladder height process and $Y$ is a L\'evy process with only negative jumps, constructed from the descending ladder height process $H^-$. One of their requirements on the descending ladder height process is to have a non-decreasing L\'evy density. In the more recent paper \cite{bib:PS13}, Patie and Savov removed the hypothesis of a non-decreasing density. Indeed, they have announced in their theorem 2.1 that such Mellin transforms always take the form of a ratio of generalized Weierstrass products. In our case, the relationship to $q$-calculus arises from the fact that the generalized Weierstrass products at hand are $q$-Pochhammer symbols.

The compound Poisson process $\zeta \log q$, although simple, is a very degenerate case which we make entirely explicit, along with its relations to $q$-calculus. Moreover, our methodology of proof, which hinges on the complex-analytic Carleson's theorem, can be easily generalized.

\subsection*{Notations}
If $E$ is a topological space, we will denote by $Bor(E)$ the $\sigma$-algebra of its Borel sets. And if $X$ is a random variable valued in a measurable space $\left( E, \Ec \right)$, $\Lc\left( X \right) = \P_X$ is the probability measure on $\left( E, \Ec \right)$ known as the distribution of $X$. Equality in law is denoted by $\eqlaw$.

We will also make use of the Vinogradov symbol:
$$ f \ll g \Leftrightarrow f = \Oc(g) \ .$$

\subsection*{Structure of the paper}
Right away, we start by stating the main theorem \ref{thm:main}, as the prerequisites for such an analytic result are minimal. Then, we justify and give the definition of the $q$-gamma random variable in subsection \ref{subsection:def_q_gamma}. This allows to recognize, hidden in the expression of $s \mapsto \E\left( \left( I^{(q, z)} \right)^s \right)$, the Mellin transform of an inverse $q$-gamma random variable. As such, the main theorem is reinterpreted in a probabilistic fashion as corollary \ref{corollary:cv_gamma}: $I^{(q, z)}$ factors as the product of two independent random variables. Finally, in subsection \ref{subsection:scaling_limit}, we formulate the weak convergence of $I^{(q, z)}$ as $q \rightarrow 1^-$, jointly with the scaling of the random walk to Brownian motion.

The proofs of theorem \ref{thm:main} and proposition \ref{proposition:cv_bm} are given separately in section \ref{section:proofs}. There, we start by explaining why the approach of \cite{bib:bby} cannot be adapted to our setting.

% --------------------------------------------------------------------
\section{Main theorem and consequences}

The $n$-th $q$-Pochhammer symbol, for $n \in \N \sqcup \{ \infty \}$ and $a \in \C$, is given by:
$$ \left(a; q \right)_{n} = \prod_{j=0}^n\left( 1 - a q^j \right) \ ,$$
and the $q$-gamma function (\cite{bib:gasper_rahman}  eq. (1.10.1)) is defined for $x>0$ by:
$$ \Gamma_q\left( x \right) := \left( 1 - q \right)^{1-x} \frac{\left(q; q \right)_{\infty}}{\left(q^x; q \right)_{\infty}} \stackrel{q \rightarrow 1^-}{\longrightarrow} \Gamma(x) \ .$$

The main result, whose proof is postponed to section \ref{section:proofs}, is the following:
\begin{thm}
\label{thm:main}
The Mellin transform of $I^{(q, z)}$ is analytic for $\Re(s)<\mu$ and is given by:
\begin{align*}
    \E\left( \left( I^{(q, z)} \right)^s \right) 
= & \Gamma\left(1 + s\right) \frac{\left( q^{1+s}; q \right)_\infty}
                                  {\left( q; q \right)_\infty}
                             \frac{ \left( z; q \right)_\infty }
                                  { \left( z q^{-s}; q \right)_\infty } \\
= & \left(1-q\right)^{-2s} \frac{ \Gamma\left( 1 + s \right) }{ \Gamma_q\left( 1+s \right) }
                           \frac{ \Gamma_q\left(\mu-s\right) }{ \Gamma_q\left( \mu \right) } \ ,
\end{align*}
where $z = q^\mu$ with $\mu>0$. Moreover, the density of the perpetuity \eqref{eq:def_perpetuity} is:
$$ \forall x \geq 0, \ i_{q,z}(x) :=  \frac{\left( z; q \right)_\infty}{\left( q; q \right)_\infty} \sum_{m=0}^\infty \sum_{n=0}^\infty \frac{(qz)^m}{ \left( q; q \right)_m } \frac{(-1)^n q^{ \binom{n}{2} } }{\left(q;q\right)_n} e^{ -x q^{m-n} } \ . $$
\end{thm}

\subsection{On the \texorpdfstring{$q$}{q}-gamma distribution}
\label{subsection:def_q_gamma}

Let $G_a$ denote a geometric random variable with parameter $0 \leq a < 1$.  Following the notations in \cite{bib:bby}, if $\left( G_{a}, G_{aq}, G_{aq^2}, \dots \right)$ are independent geometric random variables with the specified parameters, we define the random variable $R^{(q)}_a$ as:
\begin{align}
\label{eq:def_q_gamma}
R^{(q)}_a := & \left( 1- q \right)^{-1} q^{\sum_{n=0}^\infty G_{a q^n}} \ .
\end{align}
The distribution of $R^{(q)}_a$ is the $q$-analogue of the gamma distribution. Such a denomination might appear arbitrary at this point. We hope that the reader will be convinced by the following useful properties and expressions.

Remarkably, the Mellin transform of $R^{(q)}_a$ for $\Re(s) \geq 0$ is expressed thanks to $q$-exponentials. If $a=q^\kappa$, $\kappa>0$:
\begin{align}
\label{eq:mellin_q_gamma}
\E\left( \left( R^{(q)}_a \right)^s \right) = & \left( 1- q \right)^{-s} \frac{ \left( a; q \right)_\infty }{\left( a q^s; q \right)_\infty}
                                            = \frac{ \Gamma_q\left( s + \kappa \right) }{ \Gamma_q\left( \kappa \right) } \ .
\end{align}
By taking $q \rightarrow 1^-$ in equation \eqref{eq:mellin_q_gamma}, one sees the convergence in law of $R^{(q)}_{q^\kappa}$ to a gamma random variable with parameter $\kappa$. Moreover, the $q$-binomial theorem (\cite{bib:gasper_rahman} eq. (1.3.2)):
\begin{align}
\label{eq:q_binomial_thm}
\forall x \in \C, |q|<1, |z|<1, \ \frac{\left(xz; q \right)_\infty}{\left(z; q \right)_\infty} = & \sum_{n=0}^\infty \frac{\left(x; q \right)_n}{\left(q; q \right)_n} z^n
\end{align}
yields for $x=0$ and $z=a q^s$ that:
$$ \E\left( \left( R^{(q)}_a \right)^s \right) = \left( 1- q \right)^{-s} \left( a; q \right)_\infty \sum_{n=0}^\infty \frac{a^n q^{sn}}{\left(q; q \right)_n}$$
which implies, via injectivity of the Mellin transform, that:
\begin{align}
\label{eq:q_gamma_proba}
\forall n \in \N, \ \P\left( R^{(q)}_a = \left(1-q\right)^{-1} q^n \right) = & \frac{\left(a; q \right)_\infty}{\left(q; q \right)_n} a^n \ .
\end{align}
To the best of the author's knowledge, the $q$-gamma distribution was first defined by Pakes (\cite{bib:pakes} section 5), motivated by length biasing. Equation \eqref{eq:q_gamma_proba} matches indeed with the last equation of page 844 \cite{bib:pakes}. So far, we preferred a naive point of view that avoids discussing why this definition yields the ``correct'' $q$-deformation. After all, from the point of view of $q$-calculus, a $q$-gamma distribution can be defined as soon as we have a $q$-deformation of the Gamma integral given for $x>0$:
$$ \Gamma(x) = \int_{\R_+} t^{x-1} e^{-t} dt \ ,$$
which serves to define the gamma distribution. Now, the $q$-integral of a function $f: \R_+ \rightarrow \R$ is given by (1.11 in \cite{bib:gasper_rahman}):
\begin{align}
\label{eq:def_q_integral}
\forall u \in \R_+, \ \int_0^u f(t) \ d_q t & := (1-q) \sum_{n \geq 0} q^n u f\left(q^n u \right) 
\end{align}
and we have:
\begin{align}
\label{eq:q_integral}
\forall x>0, \Gamma_q(x) = \left( 1 - q \right)^{1-x} \frac{\left(q; q \right)_{\infty}}{\left(q^x; q \right)_{\infty}}
                         = \int_{0}^{\frac{1}{1-q}} t^{x-1} E_q(-t) d_q t
\end{align}
where 
$$ \forall t \in \C, \ E_q(t) := \left( t(1-q); q \right)_{\infty}$$
is the $q$-exponential. In order to see that, take $a = q^x$ in \eqref{eq:q_gamma_proba} and using the fact that probability weights sum up to $1$:
\begin{align*}
\left( 1 - q \right)^{1-x} \frac{\left(q; q \right)_{\infty}}{\left(q^x; q \right)_{\infty}} 
\quad = \quad & \left( 1 - q \right)^{1-x} \frac{\left(q; q \right)_{\infty}}{\left(q^x; q \right)_{\infty}} \sum_{n \geq 0} \P\left( R^{(q)}_{q^x} = \left(1-q\right)^{-1} q^n \right) \\
\quad = \quad & \left( 1 - q \right)^{1-x} \sum_{n \geq 0} q^{nx} \left(q^n; q \right)_{\infty}\\
\quad = \quad & \sum_{n \geq 0} q^n \left( \frac{q^{n}}{1-q} \right)^{x-1} \left(q^n; q \right)_{\infty}\\
\stackrel{eq. \eqref{eq:def_q_integral}}{=} & \int_{0}^{\frac{1}{1-q}} t^{x-1} E_q(-t) d_q t
\end{align*}
De Sole and Kac \cite{bib:desole_kac05} convincingly argue that \eqref{eq:q_integral} is indeed the ``correct'' expression among other possible candidates because of its compatibility with other formulae from $q$-calculus. As mentioned in the introduction, our approach can be seen as another argument in favor of De Sole and Kac's claim, with a probabilistic point of view.

\subsection{Wiener-Hopf factorization into independent random variables}

The random variable $I^{(q, z)}$ can be factored into the product of two independent random variables, one of which is an inverse $q$-gamma:
\begin{corollary}
\label{corollary:cv_gamma}
The following equality in law holds:
$$ I^{(q, z)} \stackrel{\Lc}{=} q^{-\sum_{n=0}^\infty G_{z q^n}} I^{(q)} = \frac{\left( 1- q \right)^{-1} I^{(q)}}{ R^{(q)}_z}$$
where $I^{(q)}$ and $R^{(q)}_z$ are independent random variables.

In particular, as $q \rightarrow 1$, $\left(1-q\right)^2 I^{(q, q^\mu)}$ converges in law to $\frac{1}{\gamma_\mu}$ with $\gamma_\mu$ a gamma random variable with parameter $\mu$.
\end{corollary}
\begin{proof}
For the equality in law, Mellin transforms of both sides are equal.

Since we already know that $R^{(q)}_{q^\mu}$ converges in law to $\gamma_\mu$ it is sufficient to prove that $(1-q) I^{(q)}$ converges to $1$ in probability, which is implied by:
$$ \E\left( (1-q) I^{(q)} \right) \stackrel{eq. \eqref{eq:def_perpetuity2}}{=}
   \sum_{n=0}^\infty (1-q) q^n \E( \varepsilon_n ) = 1 \ .$$
$$ \Var\left( (1-q) I^{(q)} \right) \stackrel{eq. \eqref{eq:def_perpetuity2}}{=}
   \sum_{n=0}^\infty (1-q)^2 q^{2n} \Var( \varepsilon_n ) = \half \frac{1-q}{1+q} \stackrel{q \rightarrow 1}{\longrightarrow} 0 \ .$$
\end{proof}

The Wiener-Hopf factorization is a general phenomenon for exponential functionals of L\'evy processes. As in the work of \cite{bib:PPS12} and \cite{bib:PS12}, the term $I^{(q)}$ appearing in the previous factorization is the exponential functional associated to the descending ladder height process of $\zeta \log q$. In this case, $I^{(q)} = I^{(q,z=0)}$ corresponds indeed to the exponential functional if the random walk $B$ only goes upward.

However, the other term i.e the inverse $q$-gamma variable is not the exponential functional of a L\'evy process. Among the possible arguments for that claim, the $q$-gamma variable is discrete, while exponential functionals are known to have a density - very degenerate cases aside (theorem 2.2 in \cite{bib:blr08}). This result provides a good example to Corollary 2.2 in \cite{bib:PS13}, where Patie and Savov give a necessary and sufficient condition for this second term to be an exponential functional.

\begin{rmk}
This Wiener-Hopf factorization in law is, to the author, quite puzzling. Indeed, the perpetuity on the left-hand side involves adding  both positive and negative powers of $q$ randomly. On the right-hand side, we see that this is tantamount \emph{in law} to paying the price of certain negative power of $q$ that is a sum of geometric random variables, then multiplying by the exponential functional associated to the descending ladder height process of $\zeta_t \log q$.

A trajectorial explanation of such an equality would certainly be desirable. The same question is raised in \cite{bib:PS13} for the Wiener-Hopf factorizations of exponential functionals associated to more general L\'evy processes.
\end{rmk}

For general L\'evy processes, the distribution of the remainder random variable $R_z^{(q)}$ is extracted from the L\'evy measure of the descending ladder process $H^-$ (see the expressions in Theorem 1.9 \cite{bib:alili}).

\subsection{Scaling limit}
\label{subsection:scaling_limit}
The thorough reference for weak convergence in sample path spaces is \cite{bib:billingsley}. The acronym càdlàg stands for ``continu à droite, limité à gauche'' in French and means right-continuous with left-limits. Let $D_\infty$ be the Skorohod space of functions on $\R$:
$$ D_\infty := \left\{ f: \R_+ \rightarrow \R \ | \ f \textrm{ is càdlàg } \right\} \ .$$
Let $d$ be the following distance which metrizes the topology of uniform convergence on compact sets:
$$ \forall (x,y) \in D_\infty \times D_\infty, \ d(x,y) := \sum_{n=1}^\infty \frac{1 \wedge d_n(x,y)}{2^n}$$
with
$$ d_n(x,y) := \sup_{0 \leq t \leq n} |x(t)-y(t)| \ .$$
As such, we will refer to $d$ as the local uniform distance. The induced topology is the (locally) uniform topology and is finer than the usual Skorohod topology on $D_\infty$. However it is not separable. There is no essential need for defining the usual Skorohod topology, but the reader is referred to section 12 in \cite{bib:billingsley} for completeness. $\Dc$ is the ball $\sigma$-field for the local uniform distance - or equivalently the Borel $\sigma$-field for the usual Skorohod topology ( See equation (15.2) in \cite{bib:billingsley}).

We are concerned with the weak convergence of the following sequence of compound Poisson processes indexed by $q$:
\begin{align}
\label{eq:def_W_q}
\left( W_t^{(q, \mu)} ; t \geq 0 \right) := & \left( \frac{-\log q}{2} \ \zeta_{ \frac{2 t}{(1-q)^2} } ; t \geq 0 \right) \ ,
\end{align}
jointly with their exponential functionals $I\left( W^{(q, \mu)} \right)$ where:
 $$\begin{array}{cccc}
    I: & D_\infty & \rightarrow & \R \sqcup \{ \infty \} \\
       &     x    & \mapsto     & 2 \int_0^\infty e^{-2 x(s)} ds \ .
   \end{array}$$

The following theorem deals with the sequence of probability measures
$$ \Lc\left( W^{(q, \mu)} , I\left( W^{(q, \mu)} \right) \right)$$
on $\left( D_\infty \times \R_+, \Dc \otimes Bor\left( \R_+ \right) \right)$. There are three reasons why we choose to work in the locally uniform topology on the Skorohod space. First, the functional $I$ is better controlled under variations of its argument's uniform norm. In the usual Skorohod topology, one would have to deal with a troublesome time rescaling as well. Secondly, our result in this setting is stronger as the uniform topology has less open sets. And thirdly, our presentation is improved by not having to define the Skorohod topology.

\begin{proposition}
\label{proposition:cv_bm}
$$\left( 1 - q \right)^2 I^{(q,z)} = I\left( W^{(q, \mu)} \right) \ .$$
and we have the weak convergence on $\left( D_\infty \times \R_+, \Dc \otimes Bor\left( \R_+ \right) \right)$:
$$ \left( W^{(q, \mu)}, I\left( W^{(q, \mu)} \right) \right) \stackrel{q \rightarrow 1^-}{\longrightarrow}
   \left( W^{(\mu)}   , I\left( W^{(   \mu)} \right) \right) \ .
$$
\end{proposition}

Putting together proposition \ref{proposition:cv_bm} and corollary \ref{corollary:cv_gamma} recovers Dufresne's identity in law \cite{bib:Dufresne90}:
$$ 2 \int_0^\infty e^{- 2 W_s^{(\mu)}} ds \eqlaw \frac{1}{\gamma_\mu} \ .$$

If one is only interested in proving that $I\left( W^{(q, \mu)} \right)$ converges weakly to $I\left( W^{(   \mu)} \right)$, which is the convergence of the second marginal in the previous theorem, one can invoke general approximation theorems such as theorem 7.3 in \cite{bib:bl13} or lemma 4.8 in \cite{bib:PPS12}. As mentioned in their text, the latter result was specifically tailored to by-pass the issue that exponential functionals are not continuous on $D_\infty$. One cannot a priori invoke the mapping theorem 2.7 in \cite{bib:billingsley}. The problem arises from the infinite time horizon: two paths may be close to each other in the locally uniform topology and yet be very different in the long run, causing their exponential functionals to be also very different. Of course, these unfavorable events occur with small probability. Not only our proof gives joint convergence, but it also shows that the mapping theorem can be applied, upon controlling such unfavorable events. The argument is robust and simple enough to generalize to other Lévy processes, although we have not explored that possibility.

\section{Proofs}
\label{section:proofs}

\subsection{Proof of theorem \ref{thm:main}}
Let us comment on the proof of \cite{bib:bby} which handles the case of $z=0$ in theorem \ref{thm:main}. It goes as follows. Consider two independent random variables $R^{(q)}_q$ and $I^{(q)}$. After computing the entire moments of $I^{(q)}$, one can notice that those of $(1-q) R^{(q)}_q I^{(q)}$ match the moments of a standard exponential random variable $\varepsilon$. Since the exponential distribution is characterized by its moments, we can deduce $(1-q) R^{(q)}_q I^{(q)} \eqlaw \varepsilon$. One concludes by taking the Mellin transform on both sides of the previous equality in law. In the proof of theorem \ref{thm:main} though, there is a small catch. If the random variable $I^{(q)}$ has entire moments, $I^{(q, z)}$ does not for $z>0$. Therefore, we will have to take a different path - through complex analysis. The crux of the argument in computing the Mellin transform is Carlson's theorem (see 9.2.1 in \cite{bib:boas}). It allows under certain conditions to identify two entire functions that match on the positive integers.

\paragraph{Mellin transform:} Recall that
$$ \E\left( e^{-s \log(q) \zeta_t} \right) = \exp\left( t \psi(s) \right)$$
where the L\'evy-Khintchine exponent is:
$$\psi(s) = \left( q^{-s} - 1 \right) + z\left( q^s - 1 \right) = \left( q^{-s} - 1 \right) \left(1-z q^s \right) \ .$$
It satisfies (8) and (9) in \cite{bib:by}, thereby confirming the existence of negative moments for $I^{(q,z)}$ (theorem 3 in \cite{bib:by}). Since a Mellin transform is always holomorphic on the interior of the vertical strip where it converges absolutely, $s \mapsto \E\left( \left( I^{(q,z)} \right)^{s} \right)$ is holomorphic on the left half of the complex plane.

Let $s = -r$, with $\Re(r)>0$. One has the recurrence formula:
\begin{align}
\label{eq:mellin_recurrence}
\E\left( \left( I^{(q,z)} \right)^{-r} \right)
= & \frac{r}{\left( q^{-r} - 1 \right) \left(1-q^{\mu+r} \right)} \E\left( \left( I^{(q,z)} \right)^{-(r+1)} \right) \ .
\end{align}
Indeed, from equation (11) in \cite{bib:by}:
\begin{align*}
\E\left( \left( I^{(q,z)} \right)^{-r} \right)
& = \frac{r}{\psi(r)} \E\left( \left( I^{(q,z)} \right)^{-(r+1)} \right)\\
& = \frac{r}{\left( q^{-r} - 1 \right) \left(1-z q^r \right)} \E\left( \left( I^{(q,z)} \right)^{-(r+1)} \right)\\
& = \frac{r}{\left( q^{-r} - 1 \right) \left(1-q^{\mu+r} \right)} \E\left( \left( I^{(q,z)} \right)^{-(r+1)} \right) \ .
\end{align*}
The idea is rather simple: Equation \eqref{eq:mellin_recurrence} identifies the Mellin transform on integers, as a meromorphic function. Then, Carleson's theorem enters the picture in order to discriminate between the different possibilities. More precisely, the recurrence allows to extend the map $r \mapsto \E\left( \left( I^{(q,z)} \right)^{-r} \right)$ to all of $\C$ except when $r \in -\mu - \N$. These are precisely the poles and the Mellin transform admits a meromorphic extension to all of $\C$. Moreover, it is holomorphic on $\{ \Re(r)>-\mu \}$. 

Let us define the meromorphic function $\varphi$ on the right half plane:
$$\varphi(r) := -1 + 
                \frac{ \E\left( \left( I^{(q,z)} \right)^{-r} \right) }
                     { \Gamma\left( 1 - r \right) }
                \frac{\left( q; q \right)_\infty}
                     {\left( q^{1-r}; q \right)_\infty}
                \frac{ \left( z q^{r}; q \right)_\infty }
                     { \left( z; q \right)_\infty } \ .$$
We are done upon proving that $\varphi$ vanishes identically. Using the recurrence \eqref{eq:mellin_recurrence}, we find that:
$$ \varphi(0) = 0 , \varphi(r) = \varphi(r+1) \ ,$$
for $r$ outside of poles. Now, by examining its constituents, the function $\varphi$ is holomorphic on the strip $-\mu < \Re(r) < 1$. Moreover, this strip is of length larger than one. Because the function $\varphi$ is $1$-periodic, it is an entire function that vanishes on the integers. In order to see it vanishes identically, we invoke Carlson's theorem as announced (theorem 9.2.1 in \cite{bib:boas}); which requires from $\varphi$ not to grow too fast on vertical strips. In order to exclude maps such as $z \mapsto \frac{\sin\left( \pi z \right)}{\pi z}$, we need to check that $\varphi$ is a function of exponential type at most $\pi$. This is handled by the following lemma.

\begin{lemma}
$\varphi$ is a function of type $\half \pi$ i.e:
$$ \forall \varepsilon>0, \exists M_\varepsilon > 0, \forall r \in \C, |\varphi(r)| \leq M_\varepsilon e^{ (\frac{\pi}{2} + \varepsilon)|\Im(r)| } \ .$$
\end{lemma}
\begin{proof}
Given that $\varphi$ is $1$-periodic, it suffices to prove the result for $\Re(r) \in \left[ \sigma_- ; \sigma_+ \right] \subset \left(-\mu; 1 \right)$, with $\left[ \sigma_- ; \sigma_+ \right]$ having a length larger than $1$. In that domain:
$$ \left| \E\left( \left( I^{(q,z)} \right)^{-r} \right) \right| \leq  \E\left( \left( I^{(q,z)} \right)^{-\sigma_-}  + \left( I^{(q,z)} \right)^{-\sigma_+} \right) \ll 1 \ ,$$
$$ \left| \left( z q^{r}; q \right)_\infty \right| \leq \left( -q^{\mu+\sigma_-}; q \right)_\infty \ll 1 \ ,$$
$$ \left| \frac{1}{\left( q^{1-r}; q \right)_\infty} \right| \leq \frac{1}{\left( q^{1-\sigma_+}; q \right)_\infty} \ll 1 \ .$$
For the two previous dominations, we respectively used the fact that the $q$-Pochhammer symbol $\left( x; q \right)_\infty$ satisfies $\left|\left( x; q \right)_\infty \right| \leq \left( -|x|; q \right)_\infty$ for $x \in \C$ and that it is decreasing for $x \in (-\infty, 1) \subset \R$. Thus:
$$   |\varphi(r)| 
\ll  1 + \left| \frac{1}{\Gamma\left( 1 - r \right)} \right| \ .$$
To conclude, we apply the complex Stirling formula ( e.g Proposition IV.1.4 in \cite{bib:freitag_busam}) with the variable going to infinity within a vertical strip. As $t \rightarrow \infty$ with $\sigma$ in a fixed compact $K$, we have:
$$ \Gamma\left( \sigma + it \right) \sim \sqrt{2 \pi} \left( i t\right)^{\sigma - \half} e^{-\frac{\pi}{2} |t|}
               \left( \frac{|t|}{e} \right)^{it} \left( 1 + O_K(\frac{1}{|t|}) \right) \ .$$
In particular:
$$ \forall \varepsilon>0, \exists M_\varepsilon>0, \forall \Re(r) \in \left[ \sigma_- ; \sigma_+ \right], |\varphi(r)| \leq M_\varepsilon e^{ (\frac{\pi}{2} + \varepsilon)|\Im(r)| } \ .$$
\end{proof}

\paragraph{Density:}
The expression for the density is obtained via inversion of the Mellin transform
\begin{align}
\label{eq:mellin_transform}
 \E\left( \left( I^{(q, z)} \right)^s \right) 
= & \Gamma\left(1 + s\right) \frac{\left( q^{1+s}; q \right)_\infty}
                                  {\left( q; q \right)_\infty}
                             \frac{ \left( z; q \right)_\infty }
                                  { \left( z q^{-s}; q \right)_\infty }. 
\end{align}
In the following, all series are absolutely convergent for $0<z<1$, $0< q < 1$ and $\Re(s)<\mu$, making the use of Fubini's theorem possible in order to exchange summations and integrals. As in \cite{bib:bby}, start by the Euler formula:
$$ \left( q^{1+s}; q \right)_\infty = \sum_{n=0}^\infty \frac{(-1)^n q^{ \binom{n}{2} } }{\left(q;q\right)_n} q^{n(1+s)}$$
together with 
$$ q^{n(1+s)} \Gamma(1+s) = \int_0^\infty x^s e^{ -x/q^n } dx$$
in order to obtain the first part of the right-hand side in equation \eqref{eq:mellin_transform}
\begin{align}
\label{eq:density_part1}
\Gamma(1+s) \frac{\left( q^{1+s}; q \right)_\infty}{\left( q; q \right)_\infty} & = \frac{1}{\left( q; q \right)_\infty} \int_0^\infty dx \ x^s \sum_{n=0}^\infty \frac{(-1)^n q^{ \binom{n}{2} } }{\left(q;q\right)_n} e^{ -x/q^n } \ .
\end{align}
Now, for the second part of the expression, an application of the $q$-binomial theorem (eq. \eqref{eq:q_binomial_thm}) yields
\begin{align}
\label{eq:density_part2}
\frac{ \left( z; q \right)_\infty }{ \left( z q^{-s}; q \right)_\infty } = & 
\sum_{m=0}^\infty \left( z; q \right)_\infty \frac{z^m q^{-ms}}{ \left( q; q \right)_m } \ .
\end{align}
Forming the product of \eqref{eq:density_part1} and \eqref{eq:density_part2} gives:
\begin{align*}
  & \E\left( \left( I^{(q, z)} \right)^s \right) \\
= \quad & \Gamma\left(1 + s\right) \frac{\left( q^{1+s}; q \right)_\infty}
                                  {\left( q; q \right)_\infty}
                             \frac{ \left( z; q \right)_\infty }
                                  { \left( z q^{-s}; q \right)_\infty }\\
= \quad & \frac{\left( z; q \right)_\infty}{\left( q; q \right)_\infty}
          \left[ \sum_{m=0}^\infty \frac{z^m q^{-ms}}{ \left( q; q \right)_m } \right]
          \left[ \int_0^\infty dx \ x^s \sum_{n=0}^\infty \frac{(-1)^n q^{ \binom{n}{2} } }{\left(q;q\right)_n} e^{ -x/q^n } \right]\\
= \quad & \frac{\left( z; q \right)_\infty}{\left( q; q \right)_\infty} \sum_{m=0}^\infty \frac{z^m}{ \left( q; q \right)_m } \int_0^\infty dx \ \left( x q^{-m} \right)^s \sum_{n=0}^\infty \frac{(-1)^n q^{ \binom{n}{2} } }{\left(q;q\right)_n} e^{ -x/q^n }\\
\stackrel{(y=x q^{-m})}{=} & \frac{\left( z; q \right)_\infty}{\left( q; q \right)_\infty} \sum_{m=0}^\infty \frac{(qz)^m}{ \left( q; q \right)_m } \int_0^\infty dy \ y^s \sum_{n=0}^\infty \frac{(-1)^n q^{ \binom{n}{2} } }{\left(q;q\right)_n} e^{ -y q^{m-n} }\\
\stackrel{(Fubini)}{=} & \int_0^\infty dy \ y^s \left[ \frac{\left( z; q \right)_\infty}{\left( q; q \right)_\infty} \sum_{m=0}^\infty \sum_{n=0}^\infty \frac{(qz)^m}{ \left( q; q \right)_m } \frac{(-1)^n q^{ \binom{n}{2} } }{\left(q;q\right)_n} e^{ -y q^{m-n} } \right] \ ,
\end{align*}
hence the result for the density.

\subsection{Proof of theorem \ref{proposition:cv_bm}}
To recognize the random variable $\left( 1 - q \right)^2 I^{(q,z)}$ as nothing but $I\left( W^{(q, \mu)} \right)$, we perform the change of variable $s = \frac{2 u}{(1-q)^2}$ in the time parameter:
$$
  \left( 1 - q \right)^2 I^{(q,z)} 
= \left( 1 - q \right)^2 \int_0^\infty q^{\zeta_s} ds 
= 2 \int_0^\infty q^{ \zeta_{ \frac{2 u}{(1-q)^2} } } du
= 2 \int_0^\infty e^{ -2 W_s^{(q, \mu)}} ds \ .
$$

And, in order to prove that for any bounded continuous functional $F: D_\infty \times \R \rightarrow \R$:
$$ \E\left( F\left( W^{(q, \mu)}, 2 \int_0^\infty e^{ -2 W_s^{(q, \mu)}} ds \right) \right) \stackrel{q \rightarrow 1^-}{\longrightarrow}
   \E\left( F\left( W^{(\mu)}, 2 \int_0^\infty e^{ -2 W_s^{(\mu)}} ds \right) \right), $$
we will approximate the exponential functional $I(\cdot)$ by the exponential functional $I_T(\cdot)$ up to a finite horizon $T>0$:
 $$\begin{array}{cccc}
    I_T: & D_\infty & \rightarrow & \R \\
       &     x    & \mapsto     & 2 \int_0^T e^{-2 x(s)} ds
   \end{array}$$
which is continuous for the locally uniform topology. Also, we adopt the convention that $W^{(q, \mu)}$ for $q=1$ refers to a Brownian motion with drift $\mu$. The rest of the proof is broken down into three steps.

\paragraph{ \bf Step 1: Weak convergence in the uniform topology on the Skorohod space}\mbox{}\\
Using Wald's identity:
$$ \E\left(   W_t^{(q, \mu)} \right) = -\frac{\log q}{2 (1-q)^2} (1+z) t \E\left( B_1 \right) = -\frac{\left( 1-q^\mu \right) \log q}{(1-q)^2} t \stackrel{q \rightarrow 1^-}{\longrightarrow} \mu t \ ,$$
$$ \Var\left( W_t^{(q, \mu)} \right) = \half \log^2 q \ (1+z) t(1-q)^{-2} \E\left( B_1^2 \right) \stackrel{q \rightarrow 1^-}{\longrightarrow} t \ .$$ 
Thanks to Donsker's invariance principle, we have weak convergence to $W^{(\mu)}$ in the usual Skorohod topology. Because the statement is usually phrased for the rescaled discrete random walk, one can adapt theorem 14.1 in \cite{bib:billingsley} to the case of a compound Poisson process: the two previous estimates on mean and variance give tightness and because of the independence of increments, the convergence of finite dimensional marginals is obtained from the convergence of Lévy exponents. We record the expression of the Lévy exponent for later use:
$$ \E\left( e^{s W_t^{(q, \mu)}} \right) = e^{t \psi_q(s)}$$
where 
\begin{align}
\label{eq:levy_khintchine_exp_q} 
\psi_q(s) = & - 2\frac{1-q^{-\half s}}{1-q} \frac{1-q^{\mu+\half s}}{1-q} \stackrel{q \rightarrow 1^-}{\longrightarrow} \half s^2 + s \mu \ .
\end{align}

At this point, we need to upgrade convergence to the finer (locally uniform) topology. This is achieved by applying theorem 6.6 in \cite{bib:billingsley}. Indeed, the Wiener measure is supported on the space of continuous functions; which is separable for the locally uniform topology. Moreover convergence in the usual Skorohod topology to a continuous function implies local uniform convergence.

Now by applying the mapping theorem (theorem 2.7 in \cite{bib:billingsley}), we have that for all $T>0$:
\begin{align}
\label{eq:weak_cv_finite_T}
 \E\left( F\left( W^{(q, \mu)}, I_T\left( W^{(q, \mu)} \right) \right) \right) & \stackrel{q \rightarrow 1^-}{\longrightarrow}
 \E\left( F\left( W^{(   \mu)}, I_T\left( W^{(   \mu)} \right) \right) \right) \ .
\end{align}

\paragraph{ \bf Step 2: Uniform control of the error in $q$}\mbox{}\\
We prove that there is a $0 < q_0 < 1$ such that for all $\delta>0$:
\begin{align}
\label{eq:uniform_control}
\limsup_{T  \rightarrow \infty} \sup_{ q_0 \leq q \leq 1}
\P\left( \left| I\left( W^{(q, \mu)} \right) - I_T\left( W^{(q, \mu)} \right) \right| \geq \delta \right) & = 0 \ .
\end{align}

As $W^{(q, \mu)}$ has independent increments:
$$ I\left( W^{(q, \mu)} \right) = I_T\left( W^{(q, \mu)} \right) + e^{-2 W^{(q, \mu)}_T } I\left( \widetilde{ W^{(q, \mu)} } \right)$$
with $\widetilde{ W^{(q, \mu)} } := W^{(q, \mu)}_{T+\cdot} - W^{(q, \mu)}_{T}$ independent from $W^{(q, \mu)}_T$. By the Markov inequality, for $0 < \nu < \mu$:
\begin{align*}
         \P\left( \left| I\left( W^{(q, \mu)} \right) - I_T\left( W^{(q, \mu)} \right) \right| \geq \delta \right)
\leq \ & \E\left[ \left(\frac{1}{\delta}\right)^\nu e^{-2 \nu W^{(q, \mu)}_T } I\left( \widetilde{ W^{(q, \mu)} } \right)^\nu \right]\\
=    \ & \left(\frac{1}{\delta}\right)^\nu e^{T \psi_q(-2 \nu) } \E\left[ I\left( W^{(q, \mu)} \right)^\nu \right] \ .
\end{align*}
Let us examine the behavior of this bound as $q \rightarrow 1$. From the Mellin transform's explicit expression in theorem \ref{thm:main}, it is clear that $\E\left( I\left( W^{(q, \mu)} \right)^\nu \right)$ is uniformly bounded as $q \rightarrow 1$. Also, from \eqref{eq:levy_khintchine_exp_q}, $\psi_q(-2 \nu) \stackrel{q \rightarrow 1^-}{\longrightarrow} 2 \nu(\nu-\mu) < 0$. Hence the existence of a $0 < q_0 < 1$ such that for all $q_0  \leq q \leq 1$, $-\psi_q(-2 \nu) \geq \nu(\mu-\nu) > 0$. Therefore:
$$   \sup_{ q_0  \leq q \leq 1} \P\left( \left| I\left( W^{(q, \mu)} \right) - I_T\left( W^{(q, \mu)} \right) \right| \geq \delta \right)
 \leq \frac{e^{ - T \nu(\mu-\nu) }}{\delta^\nu} \sup_{ q_0  \leq q \leq 1} \E\left( I\left( W^{(q, \mu)} \right)^\nu \right),
$$
which is a finite quantity. This proves the claim \eqref{eq:uniform_control} upon taking $T \rightarrow \infty$.

\paragraph{ \bf Step 3: Conclusion}\mbox{}\\
Let $\varepsilon>0$. Without loss of generality, we can assume that $F$ is uniformly continuous in the second variable, thanks to the Portmanteau theorem. Let $\delta = \delta(\varepsilon)$ be the modulus of continuity of $F$ associated to $\varepsilon>0$:
$$ \forall w \in D_\infty, \ \left( |x-y| \leq \delta \Rightarrow \left| F(w,x) - F(w,y) \right| \leq \varepsilon \right) \ .$$
We have:
\begin{align}
\label{eq:inequality_control_delta}
\forall w \in D_\infty, \ \left| F\left( w, I(w) \right) - F\left( w, I_T(w) \right) \right| & 
\leq \varepsilon + 2 |F|_\infty \mathds{1}_{\left\{ |I(x)-I_T(x)| \geq \delta \right\}}  \ .
\end{align}

By combining the triangular inequality with \eqref{eq:inequality_control_delta}, we have for all $q_0 \leq q \leq 1$ and $T>0$:
\begin{align*}
     & \left| \E\left( F\left( W^{(q, \mu)}, I\left( W^{(q, \mu)} \right) \right) \right)
           -  \E\left( F\left( W^{(   \mu)}, I\left( W^{(   \mu)} \right) \right) \right) \right|\\
%\leq & \left| \E\left( F\left( W^{(q, \mu)}, I_T\left( W^{(q, \mu)} \right) \right) \right)
%           -  \E\left( F\left( W^{(   \mu)}, I_T\left( W^{(   \mu)} \right) \right) \right) \right|\\
%     & + \left| \E\left( F\left( W^{(q, \mu)}, I  \left( W^{(q, \mu)} \right) \right) \right)
%             -  \E\left( F\left( W^{(q, \mu)}, I_T\left( W^{(q, \mu)} \right) \right) \right) \right|\\
%     & + \left| \E\left( F\left( W^{(\mu)}, I  \left( W^{(\mu)} \right) \right) \right)
%             -  \E\left( F\left( W^{(\mu)}, I_T\left( W^{(\mu)} \right) \right) \right) \right|\\
\leq & \left| \E\left( F\left( W^{(q, \mu)}, I_T\left( W^{(q, \mu)} \right) \right) \right)
           -  \E\left( F\left( W^{(   \mu)}, I_T\left( W^{(   \mu)} \right) \right) \right) \right|\\
     & + 2 \varepsilon + 2 |F|_\infty \sup_{q_0 \leq q' \leq 1} \P\left( \left| I\left( W^{(q', \mu)} \right) - I_T\left( W^{(q', \mu)} \right) \right| \geq \delta \right) \ .
\end{align*}
Now, thanks to the weak convergence \eqref{eq:weak_cv_finite_T}, the following holds for all $T>0$:
\begin{align*}
 & \limsup_{q \rightarrow 1^-}
   \left| \E\left( F\left( W^{(q, \mu)}, I\left( W^{(q, \mu)} \right) \right) \right)
       -  \E\left( F\left( W^{(   \mu)}, I\left( W^{(   \mu)} \right) \right) \right) \right|\\
\leq \ & 2 \varepsilon + 2 |F|_\infty \sup_{q_0 \leq q' \leq 1} \P\left( \left| I\left( W^{(q', \mu)} \right) - I_T\left( W^{(q', \mu)} \right) \right| \geq \delta \right) \ .
\end{align*}
Using the uniform control estimate \eqref{eq:uniform_control} and taking $T \rightarrow \infty$ in the previous expression:
$$ \limsup_{q \rightarrow 1^-}
   \left| \E\left( F\left( W^{(q, \mu)}, I\left( W^{(q, \mu)} \right) \right) \right)
       -  \E\left( F\left( W^{(   \mu)}, I\left( W^{(   \mu)} \right) \right) \right) \right|
\leq 2 \varepsilon \ .
$$
We are done as $\varepsilon>0$ is arbitrary.

% --------------------------------------------------------------------
\section*{Acknowledgments}
The author would like to thank, one the one hand, Alexander Watson and Juan Carlos Pardo for encouraging him to write this note; and on the other hand Mladen Savov for pointing out references on Wiener-Hopf factorizations.

% --------------------------------------------------------------------
\small
\bibliographystyle{halpha}
\bibliography{Bib_q_functionals}

\end{document}